\documentclass[12pt,draft]{amsart}

\usepackage[english]{babel} % English language/hyphenation
\usepackage{amsmath,amsfonts,amssymb, bm} % Math packages
\usepackage{verbatim}
\usepackage{mathrsfs}
\usepackage[cp1250]{inputenc}
\usepackage{mathtools}         % for coloneqq
\usepackage{bbm}               % for lowercase mathbb
\usepackage[baseline]{euflag}

\newtheorem{lemma}{Lemma}

\newtheorem{proposition}[lemma]{Proposition}
\newtheorem{theorem}[lemma]{Theorem}
\newtheorem{corollary}[lemma]{Corollary}
\theoremstyle{definition}
\newtheorem{definition}[lemma]{Definition}
\newtheorem{example}[lemma]{Example}
\newtheorem{remark}[lemma]{Remark}

\newcommand{\argument}       {\ignorespaces{\,\cdot\,}\ignorespaces}

\newcommand{\NN}{\mathbbm{N}}
\newcommand{\QQ}{\mathbbm{Q}}

\newcommand{\CC}{\mathbb{C}}
\newcommand{\HH}{\mathbb{H}}
\newcommand{\FF}{\mathbbm{F}}
\newcommand{\EE}{\mathbbm{E}}

\newcommand{\KK}{\mathbbm{K}}
\newcommand{\kk}{\mathbbm{k}}

\newcommand{\A}{\mathbbm{A}}
\newcommand{\B}{\mathbbm{B}}
\newcommand{\Si}{\mathbbm{S}}
\newcommand{\J}{\mathcal{I}}
\newcommand{\V}{\mathcal{V}}

\DeclareMathOperator{\TR}{Tr}

\DeclareMathOperator{\rad}{rad}
\DeclareMathOperator{\ann}{ann}

\DeclareMathOperator{\Hom}{Hom}

\newcommand{\End}{\mathrm{End}}
\newcommand{\maxLId}{\mathfrak{m}}
\newcommand{\naxLId}{\mathfrak{n}}
\newcommand{\maxId}{\mathfrak{n}}

\newcommand{\id}{\mathrm{id}}   %not a math operator because we only use it as an element of the Homs

\newcommand{\Mat}{\mathrm{M}}
\newcommand{\zentrum}{\mathcal{Z}}
\newcommand{\op}{\mathrm{op}}

\title{%
Left Jacobson rings
}

\author{Jakob Cimpri\v c}
\thanks{The first author is affiliated with the Faculty of Math. and Phys., University of Ljubljana, Slovenia, and with IMFM (Institute of Math., Phys. and Mech.)}
\thanks{The second author is affiliated with the Faculty of Math. and Phys., University of Ljubljana, Slovenia.}
\thanks{The first author acknowledges the support from grants P1-0222, J1-50002 and J1-60011 which are funded by ARIS (Slovenian Research and Innovation Agency)}
\author{Matthias Sch\" otz}
\thanks{The second author acknowledges full support from the project COMPUTE (nonCommutative polynOMial oPtimisation for qUanTum nEtworks)
which is funded by the QuantERA II Programme that has received funding from the EU's H2020 research and innovation programme~\euflag~under the GA No 101017733.}
\date{\today}

\begin{document}

\allowdisplaybreaks

\begin{abstract}
	We say that a ring is strongly (resp.\ weakly) left Jacobson if every  semiprime (resp.\ prime) left ideal
	is an intersection of maximal left ideals. There exist Jacobson rings that are not  weakly left Jacobson, e.g.\ the Weyl algebra.
	Our main result is the following one-sided noncommutative Nullstellensatz:
	For any finite-dimensional $\FF$-algebra $\A$ the ring $\A[x_1,\dots,x_n]$
	of polynomials with coefficients in $\A$ is strongly left Jacobson and every maximal left ideal
	of $\A[x_1,\dots,x_n]$ has finite codimension. We also prove that an Azumaya algebra is strongly
	left Jacobson iff its center is Jacobson and that an algebra that is a finitely generated module
	over its center is weakly left Jacobson iff it is Jacobson.
\end{abstract}

\maketitle

\section{Introduction}

\subsection{}
%Hilbert's Nullstellensatz for the $\FF$-algebra  $\A \coloneqq \FF[x_1,\dots,x_n]$ of polynomials
%over a field $\FF$ can be stated as the combination of three results: 
Hilbert's Nullstellensatz for a finitely generated commutative algebra $\A$ over a field $\FF$
is a combination of three results (see e.g.\ \cite{og}):
\renewcommand{\theenumi}{\roman{enumi}}
\begin{enumerate}
 \item \label{hilb1} Every radical ideal of $\A$ is an intersection of prime ideals.
  \item \label{hilb2} Every prime ideal of $\A$ is an intersection of maximal ideals
  (i.e. $\A$ is a Jacobson ring).
  \item \label{hilb3} Every maximal ideal of $\A$ has finite codimension.
\end{enumerate}
Geometrically, it characterizes the vanishing ideals of $\A$: A \textit{point} of $\A$ is 
an algebra homomorphism from $\A$ to a finite field extension of $\FF$.
For every set $X$ of points of $\A$ its \textit{vanishing ideal}
$\bigcap_{\phi \in X} \{a \in \A \mid \phi(a)=0\}$
is a radical ideal of $\A$. Conversely, every radical ideal $I$ of $A$ is the vanishing ideal of the set of all points $\phi$ of $\A$ that satisfy $\phi(I)=0$.

There are various ways to extend these notions to noncommutative $\FF$-algebras, which lead to
different ideas of what type of result would constitute a ``noncommutative Nullstellensatz''.
In particular, should a ``noncommutative Nullstellensatz'' deal with one-sided or two-sided ideals?

\subsection{}
The approach with two-sided ideals is already well-developed; see \cite[Ch.~9]{mcr}. 
We recall it briefly for motivation. It mimics the three steps above. 
By \cite[Thm.~10.11]{lam}, every semiprime ideal of an (associative unital) ring $\A$ is an intersection of prime ideals. This extends \eqref{hilb1}.

Motivated by \eqref{hilb2} we say that a ring $\A$ is \textit{Jacobson} 
if every prime ideal of $\A$ is an intersection of left-primitive ideals.
If $\A$ is also an $\FF$-algebra then left-primitive ideals of $\A$ coincide with 
kernels of irreducible representations of $\A$. 
Examples of Jacobson rings are:  Weyl algebras $A_n(\FF)$, 
 enveloping algebras of finite-dimensional Lie $\FF$-algebras, 
group $\FF$-algebras of polycyclic-by-finite groups; see \cite[Cor.~9.1.8]{mcr}. 
More generally, let $R$ be a commutative Jacobson ring and let $\A$ be a
constructible $R$-algebra (i.e. $\A$ is obtained from $R$ by a finite number of extensions,
each being either an almost normalizing extension or a finite module extension). 
Then $\A$ is  Jacobson by \cite[Thm.~9.4.21]{mcr}. 
Many Ore extensions are Jacobson rings; see \cite[9.7.1(b)]{mcr}.
Another example are finitely generated PI algebras over fields; 
see \cite[Th.~11.7]{formanek}. % or \cite[Cor.~1.3]{amitsur}.

The most difficult part is to find a good analogue of property \eqref{hilb3}.
For finitely generated PI algebras over fields we use the fact that every left-primitive ideal 
is maximal by \cite[Th.~4.3]{formanek} and has finite codimension by \cite[Th.~11.5]{formanek}.
In general we use the fact that every left-primitive ideal  is an intersection 
of maximal left ideals; see \cite[Cor. 11.5]{lam}.
If $\A$ is a countably generated algebra over an uncountable field $\FF$, then every maximal
left ideal $I$ of $\A$ is ``nice'' in the sense that the skew-field $\End_\A(\A/I)$ is algebraic over $\FF$;
see \cite[Def.~9.1.4 and Prop.~9.1.7]{mcr}. A similar theory exists for constructible algebras over
commutative Jacobson rings; see \cite[Def.~9.2.3 and Thm.~9.4.21]{mcr}.

The geometric meaning is similar to the above: Let $\A$ be an $\FF$-algebra.
The \textit{points} of $\A$ are ``nice'' irreducible representations of $\A$ 
(i.e. the endomorphism ring of the corresponding simple
$\A$-module is algebraic over $\FF$).
For every set $X$ of points of $\A$ its \textit{vanishing} ideal
\begin{equation*}
  \bigcap\nolimits_{\pi \in X}~\{ a \in \A \mid \pi(a) = 0 \}
\end{equation*}
is semiprime. Conversely, every semiprime ideal of $\A$ is of this form.

\subsection{}
We are interested in a different approach, dealing only with one-sided ideals.
We say that an $\FF$-algebra  \textit{satisfies the left Nullstellensatz} if it satisfies
the following properties motivated by  \eqref{hilb1}-\eqref{hilb3}:
 \renewcommand{\theenumi}{\alph{enumi}}
\begin{enumerate}
\item \label{left1} Every semiprime left ideal is an intersection of prime left ideals.
\item \label{left2} Every prime left ideal is an intersection of maximal left ideals.
\item \label{left3} Every maximal left ideal has finite codimension.
\end{enumerate}
In Section \ref{sec2} we will explain the geometric meaning of such algebras:
A \textit{nice directional point} of an $\FF$-algebra $\A$ is a pair $(\pi,v)$ where 
$\pi$ is a finite-dimensional irreducible representation of $\A$ and $v$
is an element of the representation space of $\pi$. For every set $X$ of nice directional points
of $\A$ the set 
\begin{equation*}
  \bigcap\nolimits_{(\pi,v) \in X}~\{ a \in \A \mid \pi(a)v = 0 \}
\end{equation*}
is a semiprime left ideal of $\A$. Conversely, if $\A$ satisfies the left Nullstellensatz,
then every semiprime left ideal  is of this form for some $X$.

We say that a ring is \textit{weakly left Jacobson} if it satisfies \eqref{left2}.
Note that every weakly left Jacobson ring is Jacobson.
In Section \ref{sec3} we give examples of  Jacobson rings that are not weakly left Jacobson.
The main examples are the Weyl algebra $A_1(\FF)$
and the universal enveloping algebra of the Heisenberg-Lie algebra.
In Section \ref{sec4} we prove that a Jacobson ring that is finitely generated as a module over its center is weakly left Jacobson, see Corollary~\ref{corollary:Ibot}.

We say that a ring is \textit{strongly left Jacobson} if it satisfies \eqref{left1} and \eqref{left2}.
To the best of our knowledge, it is an open question whether (a) is true for every ring.
Known examples of strongly left Jacobson rings are discussed in Section \ref{sec1}.
%Known examples of strongly left Jacobson rings are the matrix algebras $\mathrm{M}_n(R)$
%where $R$ is a commutative Jacobson ring (see \cite[Thm.~3]{c2})
%and the standard quaternion algebra $\mathbb{H}_R$ over $R$, 
%where $R$ is a commutative Jacobson ring containing $\frac12$
%(see  \cite[Thm.~1.2]{c3} which is based on \cite{ap2}).
In Section \ref{sec5} we show that an Azumaya algebra is strongly left Jacobson if and only if its center is Jacobson,
see Theorem~\ref{theorem:stronglyLeftJacobson}.

In Section \ref{sec6} we prove our main result (Theorem~\ref{theorem:algebraValuedPolynomials})
which says that for any finite-dimensional $\FF$-algebra $\A$ and central variables
$x_1,\ldots,x_n$, the $\FF$-algebra $\A[x_1,\ldots,x_n]$ satisfies the left Nullstellensatz.

\section{A brief survey of left Nulstellens\" atze}
\label{sec1}

\subsection{Quaternionic Nullstellensatz}

A point $a=(a_1,\ldots,a_d) \in \HH^d$ is \textit{central}  iff $a_i a_j=a_j a_i$ for all $i$ and $j$.
Let $\HH[x_1,\ldots,x_d]$ be the ring of all quaternionic polynomials in $d$ central variables.
For every 
\begin{equation}
\label{def1}
f =\sum_{i_1,\ldots,i_d} c_{i_1,\ldots,i_d} x_1^{i_1} \cdots x_d^{i_d} \in \HH[x_1,\ldots,x_d]
\end{equation}
and every central point $a=(a_1,\ldots,a_d) \in \HH^d$, we define the value 
\begin{equation}
\label{def2}
f(a)=\sum_{i_1,\ldots,i_d} c_{i_1,\ldots,i_d} a_1^{i_1} \cdots a_d^{i_d} \in \HH.
\end{equation}

In 2021 G. Alon and E. Paran characterized maximal left ideals of $\HH[x_1,\ldots,x_d]$
as sets of the form $I_a:=\{ g \in \HH[x_1,\ldots,x_d] \mid g(a)=0\}$ where $a \in \HH^d$ is a central point;
see \cite[Theorem 1.1]{ap2}. They also showed that every completely prime left ideal of $\HH[x_1,\ldots,x_d]$
is an intersection of maximal left ideals; see \cite[Proposition 4.4]{ap2}. 

In 2025 J. Cimpri\v c showed that $\HH[x_1,\ldots,x_d]$ is a strong left Jacobson ring.
He also extended this result
from $\HH[x_1,\ldots,x_d]$ to  the standard quaternion algebra $\mathbb{H}_R=\left( -1,-1 \atop R \right)$ 
where $R$ is a commutative Jacobson ring containing $\frac12$; see \cite[Theorems 2.1 and 3.1]{c3}.

In 2024 M. Aryapoor proved an explicit version of Quaternionic Nullstellensatz; see \cite[Theorem 1.5]{aya}.
In 2026 he extended this result (and the characterization of maximal left ideals) from $\HH[x_1,\ldots,x_d]$ to 
$D[x_1,\ldots,x_d]$ where $D$ is a centrally algebraically closed skew-field; see \cite[Theorems 4.5 and 4.7]{cac}.
We do not know whether the strong left Jacobson property extends, too.

Theorem \ref{thmb} summarizes the discussion above for quaternionic polynomials.
We are mostly interested in the generalizations of (b) $\Leftrightarrow$ (c).
To generalize (a) $\Leftrightarrow$ (b) one has to replace the central points
with the directional points; see Section \ref{sec2}. 
The relation between the central points and the directional points for $\HH[x_1,\ldots,x_d]$ 
was discussed in the proof of  \cite[Theorem A.3., $(1) \Rightarrow (2)$]{c3}.

\begin{theorem}
\label{thmb}
Let $P$ be a subset of $\HH[x_1,\ldots,x_d]$.
For every element $q \in \HH[x_1,\ldots,x_d]$ the following are equivalent.
\begin{enumerate}
\item Every central point of $\HH^d$ that annihilates all elements of $P$ also annihilates $q$.
\item $q$ is in the intersection of all maximal left ideals that contain $P$.
\item $q$ is in the smallest semiprime left ideal that contains $P$.
\item For every $c  \in \HH$ there exists $N \in \NN_0$ such that 
$$(cq)^N \in I+I(cq)+\ldots+I(cq)^N.$$
where  $I$ is the left ideal of $\HH[x_1,\ldots,x_d]$ generated by $P$.
\end{enumerate}
\end{theorem}

Some  geometric aspects of Quaternionic Nullstellensatz are discussed in \cite{ap3,ap4,slice,slice2}.

\subsection{Matrix Nullstellensatz}

In 2022 J. Cimpri\v c proved that for every commutative Jacobson ring $R$ and every $n \in \NN$
the ring $M_n(R)$ is strongly left Jacobson; see \cite[Theorem 3]{c2}. 
In 2025 he also proved an explicit version of Matrix Nullstellensatz; see \cite[Theorem 5.1]{c3}.
A characterization of maximal left ideals in $M_n(R)$ was already proved by 
D. R. Stone in 1980; see \cite[Theorem 1.2]{stone}.

Theorem \ref{thmd} summarizes these results for matrix polynomials
over algebraically closed fields:

\begin{theorem}
\label{thmd}
Let $P$ be a subset of $M_n(\kk[x_1,\ldots,x_d])$ where $\kk$ is an algebraically closed field and $n,d \in \NN$.
For every $q  \in M_n(\kk[x_1,\ldots,x_d])$ the following are equivalent:
\begin{enumerate}
\item For every  $a \in \kk^d$ and $v \in \kk^n$ (i.e. a directional point $(a,v)$)
 such that $p(a)v=0$ for every $p \in P$ we have $q(a)v=0$;
\item $q$ is in the intersection of all maximal left ideals that contain $P$.
\item $q$ is in the smallest semiprime left ideal that contains $P$.
\item For every $c \in M_n(\kk)$ there exists $N \in \NN_0$ such that 
$$(cq)^N \in I+I(cq)+\ldots+I(cq)^N.$$
where  $I$ is the left ideal of $M_n(\kk[x_1,\ldots,x_d])$ generated by $P$.
\end{enumerate}
\end{theorem}

\subsection{Bergman's Nullstellensatz}

The original motivation for this paper was the following result:

\begin{theorem} \cite[Theorem 6.3]{bergman} Let $\A$ be the free real algebra in variables $\{x_1,...,x_g\}$. Let $P$ be a finite subset of $\A$ and let $q$ be a given element of $\A$.
The following are equivalent:
\begin{enumerate}
\item For every finite-dimensional representation $\pi \colon \A \to \End(V)$ and for every vector $v \in V$ such that
$\pi(p)v=0$ for every $p \in P$ we have $\pi(q)v=0$.
\item $q$ is in the left ideal of $\A$ generated by $P$.
\end{enumerate}
 In (a) it suffices to consider representations of dimension at most  $\sum_{j=0}^d g^j$ where $d$ is the maximum of the $\deg(q)$ and $\{\deg(p): p \in P\}$.
\end{theorem}
In \cite[Theorem 4.1]{bk} it was observed that this result also holds 
for finitely generated free algebras over any field.

It is still an open question whether the following are  equivalent:
\begin{enumerate}
\item[(a')] For every finite-dimensional irreducible representation $\pi \colon \A \to \End(V)$ 
and for every vector $v \in V$ such that $\pi(p)v=0$ for every $p \in P$ we have $\pi(q)v=0$.
\item[(b')] $q$ is in the smallest semiprime left ideal of $\A$ that contains $P$.
\end{enumerate}

\section{Geometric interpretation}
\label{sec2}

We assume that all rings and algebras are associative and unital. Two-sided ideals will simply be called ideals. 

We will use the definitions of prime and semiprime left ideals from \cite[Sec.~1]{hansen}.
We say that a left ideal $I$ of a ring $\A$ is \emph{prime}
if for every $a,b \in \A$ such that $a\A b \subseteq I$ we have $a \in I$ or $b \in I$.
We say that $I$  is \emph{semiprime} if for every $a \in \A$ such that $a\A a \subseteq I$ we have $a \in I$.
Note that every maximal left ideal is prime, every prime left ideal is semiprime, and any intersection of semiprime left ideals is again semiprime.

Pick a left ideal $I$ of a ring $\A$. We define its \emph{Jacobson radical}
$\rad(I)$ as the intersection of all maximal left ideals that contain $I$.
%For two-sided ideals this coincides with the definition above by \cite[Cor.~11.5]{lam}.
For an algebraic characterization of $\rad(I)$ see
\cite[Expl.~1.4]{superfluous}.
% and \cite[Theorem 2.2]{closures}.
We define the \emph{semiprime radical} $\sqrt[st]{I}$ of $I$ as the smallest semiprime left ideal containing $I$
and the \emph{prime radical} $\sqrt[wk]{I}$ of $I$ as the intersection of all prime left ideals containing $I$.
As every prime left ideal is semiprime, $\sqrt[st]{I} \subseteq \sqrt[wk]{I}$.
We do not know whether $\sqrt[st]{I} = \sqrt[wk]{I}$ in general, but we give sufficient conditions
for this to be true in Corollary~\ref{corollary:Ibot}.
Since every maximal left ideal is prime, we have $\sqrt[wk]{I} \subseteq \rad(I)$ for every left ideal $I$.
Clearly, $\A$ is weakly (resp. strongly) left Jacobson if and only if $\sqrt[wk]{I} = \rad(I)$ (resp. $\sqrt[st]{I} = \rad(I)$) for every left ideal $I$ of $\A$.

Lemma \ref{interpretationmax} gives a geometric interpretation of maximal left ideals:

\begin{lemma} \label{interpretationmax}
  Let $\A$ be a ring, $M$ a simple left $\A$-module and $m$ a nonzero element of $M$.
  Then the annihilator $\ann_M(m)\coloneqq \{r \in \A \mid rm=0\}$ of $m$ is a maximal left ideal of $\A$.
  Moreover, every maximal left ideal of $\A$ is of this form.

  In particular, if $\A$ is an $\FF$-algebra, then maximal left ideals of $\A$ coincide with
  left ideals of the form $J_{\pi,v}\coloneqq \{a \in \A \mid \pi(a)v=0\}$ where $\pi \colon \A \to \End(V_\pi)$
  is an irreducible representation of $\A$ and $v \in V_\pi \setminus \{0\}$.
\end{lemma}

\begin{proof}
  The annihilator $\ann_M(m)$ is clearly a left ideal of $\A$.
  To show that it is maximal we have to prove that for every $c \in \A \setminus \ann_M(m)$,
  $\ann_M(m)+\A c=\A$. If $c \in \A \setminus \ann_M(m)$, then $cm \ne 0$ which implies
  that $\A cm=M$ since $M$ is simple. In particular, there exist $r \in \A$ such that $m=rcm$;
  i.e.\ $1-rc \in \ann_M(m)$.

  To prove the converse pick a maximal left ideal $I$ of $\A$ and note
  that $M\coloneqq \A/I$ is a simple left $\A$-module containing $m\coloneqq 1+I$.
  Clearly, $\ann_M(m)=I$.
\end{proof}

Let us define a \textit{directional point} of $\A$ as a pair $(M,m)$ where $M$ is a simple left $\A$-module
and $m \in M \setminus \{0\}$ (or in the case of $\FF$-algebras as a pair $(\pi,v)$ where
$\pi \colon \A \to \End(V_\pi)$ is an irreducible representation of $\A$ and $v \in V_\pi \setminus \{0\}$).
We will consider the elements of $\A$ as ``polynomials''.
We say that  a ``polynomial'' $a \in \A$ \textit{annihilates} a directional point $(M,m)$ (resp. $(\pi,v)$) 
if $am=0$ (resp. $\pi(a)v=0$).

For every set $X$ of directional points of $\A$ we write $\J(X)$ for 
the set of all ``polynomials'' from $\A$ that annihilate all elements of $X$.
If $X$ is a singleton  then $\J(X)$ is a maximal left ideal by Lemma \ref{interpretationmax}.
It follows that for every set of directional points $X$, 
the set $\J(X)=\bigcap_{x\in X} \J(x)$ is a semiprime left ideal.

For every set of ``polynomials''  $G \subseteq \A$ we write  $\V(G)$ for the set 
of all directional points  of $\A$ that are annihilated by all elements of $G$.
The set
$\J(\V(G))=\bigcap_{x \in \V(G)} \J(x)=\bigcap \{ \J(x) \mid x \text{ such that } G \subseteq \J(x) \}$
is %by Lemma \ref{interpretationmax}
equal to the intersection of all maximal left ideals of $\A$ that contain $G$.
Therefore, we  get the following geometric interpretation of the 
strong left Jacobson property (analogously for the weak version):

\begin{proposition} \label{interpretation2}
  A ring (or $\FF$-algebra) $\A$ is strongly left Jacobson if and only if for every set of ``polynomials'' 
  $G \subseteq \A$   the set $\J(\V(G))$ is equal to the smallest semiprime left ideal of $\A$ that contains $G$.
\end{proposition}

Let $\A$ be an $\FF$-algebra.
We would prefer to consider only those directional points $(\pi,v)$ of $\A$
which are ``nice'' in the sense that $\pi$ is finite-dimensional.
For every set of  ``polynomials'' $G \subseteq \A$ we write $\V_{\rm nice}(G)$
for the set of all nice directional points of $\A$ that are annihilated by all elements of $G$.

The set $\J(\V_{\rm nice}(G))$ is equal to the intersection of all maximal left ideals
of finite codimension that contain $G$. Namely, note that by the proof of Lemma \ref{interpretationmax},
maximal left ideals of finite codimension coincide with the left ideals
of the form $J_{\pi,v}$ where $\pi$ is a finite-dimensional irreducible 
representation and $v \in V_\pi \setminus \{0\}$.

Recall that by definition an $\FF$-algebra $\A$ satisfies the left Nullstellensatz
if it is strongly left Jacobson and every maximal left ideal has finite codimension.
By the above, we have the following geometric  interpretation:

\begin{proposition} \label{interpretation3}
  An $\FF$-algebra $\A$ satisfies the left Nullstellensatz if and only if for every set of ``polynomials''
  $G \subseteq \A$ the set $\J(\V_{\rm nice}(G))$ is equal to the smallest semiprime left ideal of $\A$
  that contains $G$.
\end{proposition}

\section{Counterexamples}
\label{sec3}

The aim of this section is to give examples of $\FF$-algebras that satisfy the Jacobson property but do not satisfy the weak left Jacobson property.

\begin{theorem} \label{weyl1}
  For any field $\FF$ of characteristic $0$, the first Weyl algebra
  \begin{equation}
    A_1(\FF)=\FF\langle x,y \rangle/(yx-xy-1)
  \end{equation}
  is not weakly left Jacobson.
\end{theorem}

\begin{proof}
  We show first that every left ideal $I$ of $\A $ is prime, so $I=\sqrt[wk]{I}$.
  Namely, if $a,b\in \A$ fulfil $a\A b \subseteq I$, then also $(\sum \A a\A )b \subseteq I$. Since $\A $ is simple
  we have either $\sum \A a\A =0$ or $\sum \A a\A =\A $. In the first case $a=0 \in I$, and
  in the second case $1 \in \sum \A a\A $ implies $b \in I$.

  We now take the left ideal $I\coloneqq \A yx$. Clearly $x \notin I$.
  To show that $I \ne \rad(I)$ it suffices to show that $x \in \maxLId$ for every maximal left ideal $\maxLId$ of $\A$ with $yx \in \maxLId$.
  So assume to the contrary that $\maxLId$ is a maximal left ideal of $\A$ and that $yx \in \maxLId$ but $x\notin \maxLId$.
  By maximality of $\maxLId$ there exists $r\in \A$ such that $1+rx \in \maxLId$. We can expand $r$ as
  $r=\sum_{i=0}^\ell a_i(x) y^i \in \A $ with suitable $\ell\in\NN_0$ and $a_i(x) \in \FF[x]$ for $i\in \{0,\dots,\ell\}$.
  Note that $y x^k-x^k y=k x^{k-1}$ for every $k \in \NN$, and $y x^0 - x^0 y = 0$.
  Write $r^{(1)} \coloneqq yr-ry=\sum_{i=0}^\ell a_i^{(1)}(x) y^i$ where $a_i^{(1)}(x)$ is the first derivative of $a_i(x)$.
  Using analogous notation for arbitrary derivatives, we claim that $y^n +r^{(n)} x \in \maxLId$ for every $n \in \NN_0$. The case $n=0$ follows from the definition of $r$.
  Suppose now that the claim is true for some $n\in \NN_0$, then $y^{n+1}+r^{(n+1)} x=y^{n+1}+(y r^{(n)}-r^{(n)} y) x
  =y(y^n+r^{(n)} x)-r^{(n)} y x \in \maxLId$ since $\maxLId$ is a left ideal containing $yx$; so the claim is also true for $n+1$.
  If $n$ is larger than the maximum of degrees of $a_i(x)$, the claim implies that $y^n \in \maxLId$.

  So we know that there exists $n\in \NN_0$ such that $y^n \in \maxLId$. We now show that in fact $y^n \in \maxLId$ for all $n\in \NN_0$.
  For this it remains to show that for any $k \in \NN$ such that $y^k \in \maxLId$ we also have $y^{k-1} \in \maxLId$.
  Namely, $k y^{k-1}=y^k x-x y^k=y^{k-1} yx - x y^k \in \maxLId$ since $yx \in \maxLId$ and $y^k \in \maxLId$;
  as $\FF$ has characteristic $0$ this shows that indeed $y^{k-1} \in \maxLId$.
  In particular $1 \in \maxLId$, so we have obtained a contradiction.
\end{proof}

To deduce other counterexamples we need the following:

\begin{lemma} \label{lemma:homomorphicimage}
 Let $\pi \colon \A \to \B$ be a surjective ring homomorphism. The image and preimage under $\pi$
 give a bijective correspondence between the left ideals of $\B$ and the left ideals of $\A$ 
 that contain $\ker \pi$. It also preserves semiprime, prime and maximal left ideals.
 In particular if $\A$ is weakly or strongly left Jacobson, then so is $\B$.
\end{lemma}

\begin{proof}
  It is easy to check that the image and preimage under $\pi$ give a bijective correspondence between the left ideals of
  $\B$ and the left ideals of $\A$ that contain $\ker \pi$. This correspondence clearly preserves the order by inclusion,
  so it preserves maximal left ideals. In the following we check that this correspondence also preserves semiprime left
  ideals; the argument for prime left ideals is analogous.

  Let $J$ be a semiprime left ideal of $\B$. To prove that the left ideal $\pi^{-1}(J)$ is semiprime, pick $a \in \A$
  such that $a \A a \subseteq \pi^{-1}(J)$. It follows that $\pi(a) \pi(\A) \pi(a) \subseteq J$.
  Since $J$ is semiprime and $\pi(\A)=\B$, it follows that $\pi(a) \in J$. Therefore $a \in \pi^{-1}(J)$.

  Let $I$ be a semiprime left ideal of $\A$ containing $\ker \pi$.
  To show that $\pi(I)$ is semiprime, pick $b \in \B$ such that $b\B b \subseteq \pi(I)$.
  Since $\pi$ is surjective, we have $\B=\pi(\A)$ and $b=\pi(a)$ for some $a \in \A$.
  It follows that $a \A a \subseteq \pi^{-1}(\pi(I))=I$ since $\ker \pi \subseteq I$, and so $a \in I$ since $I$ is semiprime. Therefore $b \in \pi(I)$.

  Finally assume that $\A$ is weakly (strongly) left Jacobson and let $J$ be a semiprime (prime) left ideal of $\B$.
  Then $\pi^{-1}(J)$ is a semiprime (prime) left ideal of $\A$, so $\pi^{-1}(J) = \bigcap M$ where $M$ is a set of
  maximal left ideals of $\A$. For every $\maxLId \in M$ clearly $\ker \pi \subseteq \pi^{-1}(J) \subseteq \maxLId$,
  so $\pi(\maxLId)$ is a maximal left ideal of $\B$. It follows that $J=\bigcap \{ \pi(\maxLId) \mid \maxLId \in M\}$.
  This shows that $\B$ is weakly (strongly) left Jacobson.
\end{proof}

As a special case of Lemma \ref{lemma:homomorphicimage} we obtain:

\begin{corollary} \label{weyl2}
  Any algebra over a field of characteristic $0$ which has $A_1(\FF)$ as a homomorphic image is not weakly left Jacobson.
\end{corollary}

Let us give some special cases of Corollary \ref{weyl2}.

\begin{example} \label{counter2}
  The Heisenberg Lie algebra $\mathfrak{h}$  is the  Lie algebra over a field $\FF$ with basis
  $X, Y, Z \in \mathfrak{h}$ and with the Lie bracket satisfying $[X , Y ] = Z$ and $[X , Z ] = [ Y , Z ] = 0$.
  We have a homomorphism $\Phi$ from the enveloping algebra $U(\mathfrak{h})$ onto the Weyl algebra $A_1(\FF)$
  defined by $\Phi(X)=x$, $\Phi(Y)=y$ and $\Phi(Z)=-1$. If $\FF$ has characteristic $0$, then $U(\mathfrak{h})$ is not left Jacobson.
\end{example}

\begin{example}
  Since $A_1(\FF)=\FF\langle x,y \rangle/(yx-xy-1)$ is a homomorphic image of $\FF\langle x,y \rangle$,
  the free $\FF$-algebra $\FF\langle x,y \rangle$ is not left Jacobson if $\FF$ has characteristic $0$.
\end{example}

Recall that by \cite[Cor~9.1.8]{mcr} the algebras from Theorem~\ref{weyl1} and Example~\ref{counter2} are Jacobson.

\section{Some weakly Jacobson rings}
\label{sec4}

In this section we show that a ring that is finitely generated as a module over its center is weakly left Jacobson
if and only if it is Jacobson; see Corollary~\ref{corollary:Ibot}.

Consider any left ideal $I$ of a ring $\A$ with center $R \coloneqq \zentrum(\A)$. Set
\begin{equation}
  \label{eq:quotient}
  [I : \A] \coloneqq \{ a \in \A \mid a \A \subseteq I \}
  .
\end{equation}
This is a particular example of an ideal quotient. It is easy to check that $[I : \A]$ is a
two-sided ideal of $\A$, and $[I : \A] \subseteq I$ because $\A$ has a unit.

\begin{proposition} \label{proposition:sim}
  Let $\A$ be a ring with center $R \coloneqq \zentrum(\A)$ and $I$ a prime left ideal of $\A$.
  Then $[I:\A]$ is a (two-sided) prime ideal of $\A$ and $I \cap R$ is a prime ideal of $R$.
  Moreover,  $I \cap R = [I:\A] \cap R$, which implies that the inclusion mapping $R \subseteq \A$
  induces an injective mapping of the integral domain $R/(I \cap R)$ into the center of the prime ring $\A/[I:\A]$.
\end{proposition}
\begin{proof}
  Consider $a,a' \in \A$ such that $a \A a' \subseteq [I:\A]$.
  If $a' \notin [I:\A]$, then there exists $b' \in \A$ such that $a'b' \notin I$.
  However, $abca'b' \in I$ for all $b,c \in \A$ because $abca' \in [I:\A]$,
  and therefore $ab \in I$ for all $b\in \A$ because $I$ is prime, i.e.\ $a \in [I:\A]$.
  This shows that $[I:\A]$ is a prime ideal of $\A$.

  Next consider $r,s \in R$ such that $rs \in I$. Then $r\A s = \A rs \subseteq I$,
  so $r\in I$ or $s\in I$ because $I$ is prime. This shows that $I \cap R$ is a prime ideal of $R$.
  
Since $[I:\A]  \subseteq I$, we have $[I:\A] \cap R \subseteq I \cap R$. To prove the converse pick 
$r \in I \cap R$ and note that $r \A=\A r \subseteq I$ which implies that $r \in   [I:\A]$.
\end{proof}

\begin{definition}
\label{def:ai}
  Let $\A$ be a ring with center $R \coloneqq \zentrum(\A)$ and $I$ a prime left ideal of $\A$. By Proposition \ref{proposition:sim}
  we can identify the integral domain $\bar{R}\coloneqq R/(I \cap R)$ with a subset of the center of the prime ring $\bar{\A}\coloneqq \A/[I:\A]$.
  Write $S\coloneqq \bar{R} \setminus \{0\}$ and note that the field $R_I\coloneqq S^{-1} \bar{R}$ is contained in the center of the prime ring
 $\A_I \coloneqq  S^{-1} \bar{\A}$. For every $a \in \A$ and $q \in R \setminus I$ we write $a/q \coloneqq ( \bar{q})^{-1} \bar{a} \in \A_I$ where
 $\bar{a}$ is the projection of $a$ to $\bar{\A}$ and $\bar{q}$ is the projection of $q$ to $\bar{R}$. 
\end{definition}

\begin{remark}
\label{rem:ai}
If $a,b \in \A$ and $q,s \in R \setminus I$ then
$a/q + b/s =(sa+qb)/(qs)$ and $(a/q)(b/s) = (ab)/(qs)$.
Moreover, $a/q=b/s$  iff $sa-qb \in [I:\A]$.
To see that $\A_I$ is indeed prime assume that $(a/q)(x/1)(b/s)= 0/1$
for every $x \in \A$. It follows that $axb \in [I:\A]$ for all $x\in \A$. Since $[I:\A]$
 is a prime ideal, $a\in [I:\A]$ or $b\in [I:\A]$. So $a/q = 0/1$ or $b/s = 0/1$.
\end{remark}

\begin{proposition} \label{proposition:genI}
  Let $\A$ be a ring with center $R \coloneqq \zentrum(\A)$ and $I$ a prime left ideal of $\A$.
  Assume $I \neq \A$. Then
  \begin{align}
    \langle I \rangle
    \coloneqq
    {}&\bigl\{ a/q \mid a\in I~\text{and}~q\in R \setminus I \bigr\}
  \intertext{is a left ideal of $\A_I$ and}
    \label{eq:genIintersection}
    I
    =
    {}& \bigl\{ a\in \A \mid a/1 \in \langle I \rangle \bigr\}
    .
  \end{align}
\end{proposition}
\begin{proof}
  Given $a/q, b/s \in \langle I \rangle$ with $a,b \in I$ and $q,s \in R \setminus I$,
  then $a/q+b/s = (sa + qb)/(qs) \in \langle I \rangle$, so $\langle I \rangle$ is closed under
  addition. Given $a/q \in \A_I$ and $b/s \in \langle I \rangle$ with $a \in \A$, $b \in I$,
  and $q,s \in R \setminus I$, then $(a/q)(b/s) = (ab)/(qs) \in \langle I\rangle$, so $\langle I \rangle$
  is a left ideal of $\A_I$.

  The inclusion ``$\subseteq$'' in \eqref{eq:genIintersection} holds by definition of $\langle I \rangle$. Conversely,
  suppose $a \in \A$ fulfils $a/1 \in \langle I \rangle$. This means  there are $b\in I$
  and $s\in R \setminus I$ such that $a/1 = b/s$; i.e.\ $sa - b \in [I:\A]$. Since $[I:\A] \subseteq I$ and $b \in I$, it follows that $sa \in I$
  and so $s\A a = \A sa \subseteq I$. Therefore $a \in I$ because $I$ is prime and $s\notin I$.
\end{proof}

Recall that any nondegenerate bilinear form $\beta \colon V \times V \to \FF$ on a finite-dimensional $\FF$-vector space $V$
induces an $\FF$-linear isomorphism between $V$ and its dual vector space $V^*$, namely $\argument^\flat \colon V \to V^*$,
$w\mapsto w^\flat$ with $w^\flat(v) \coloneqq \beta(v,w)$ for all $v,w\in V$. In this case it follows that for every
linear subspace $U$ of $V$ and any $v\in V \setminus U$ there exists $w\in V$ such that $\beta(u,w) = 0$ for all $u \in U$ and $\beta(v,w) \neq 0$.

\begin{theorem} \label{theorem:Ibot}
  Let $\A$ be a ring that is finitely generated as a module over its center
  $R \coloneqq \zentrum(\A)$ and $I$ a prime left ideal of $\A$.
  Write
  \begin{align}
    I^\perp \coloneqq{}&  \bigl\{ b \in \A \mid I b \subseteq [I : \A ]  \bigr\}.
  \intertext{Then $I^\perp$ is a right ideal of $\A$ and}
    \label{eq:Ibot}
    I ={}& \bigl\{ a \in \A \mid a I^\perp \subseteq [I : \A ] \bigr\}.
  \end{align}
\end{theorem}

\begin{proof}
  It is easy to check that $I^\perp$ is a right ideal of $\A$ and that the inclusion ``$\subseteq$'' holds in \eqref{eq:Ibot},
  so it only remains to prove the converse inclusion ``$\supseteq$''.
  Assume $c \in \A \setminus I$, then in order to complete the proof we have to find $b\in I^\perp$
  such that $cb \notin [I : \A]$. Note that $I \neq \A$.

  By Definition~\ref{def:ai} and Remark \ref{rem:ai}, $\A_I$ is a prime $R_I$-algebra. 
  As $\A$ is finitely generated as an $R$-module by assumption, $\A_I$ has finite dimension over $R_I$,
  i.e.\ $\A_I$ is a finite-dimensional prime algebra over the field $R_I$.
  By \cite[Prop.~11.7]{lam} this implies that $\A_I$ is simple, so $\A_I$ is central simple over its center $\zentrum(\A_I) \supseteq R_I$.
  By \cite[Lemma~1.1]{lewis} and \cite[Lemma 4 and formulas (9) and (11) in $\S$22]{draxl},
  the reduced trace $\TR \colon \A_I \to \zentrum(\A_I)$ is a nondegenerate trace functional;
  i.e.\ $\TR$ is a $\zentrum(\A_I)$-linear functional that fulfils $\TR(xy) = \TR(yx)$ for all $x,y\in \A_I$,
  and for every $x\in \A_I \setminus \{0/1\}$ there exists $y\in \A_I$ such that $\TR(xy) \neq 0/1$.

  As $c \in \A \setminus I$ it follows from \eqref{eq:genIintersection} of the previous Proposition~\ref{proposition:genI}
  that $c/1 \in \A_I \setminus \langle I \rangle$. As $\langle I \rangle$ is a left ideal of $\A_I$ it is in particular a
  $\zentrum(\A_I)$-linear subspace. As the bilinear functional $\A_I \times \A_I \ni (x,y) \mapsto \TR(xy) \in \zentrum(\A_I)$
  is nondegenerate and as $\A_I$ is finite-dimensional over $R_I$, hence also over $\zentrum(\A_I)$,
  there exists $y \in \A_I$ such that $\TR(xy) = 0/1$ for all $x\in \langle I \rangle$
  and $\TR\bigl((c/1)y\bigr) \neq 0/1$.
  The former means that $xy = 0/1$ for all $x\in \langle I \rangle$,
  because $\TR$ is nondegenerate and $\TR(xyz) = \TR(zxy) = \TR(x'y) = 0/1$ for all $x\in \langle I \rangle$ and all $z\in \A_I$
  with $x' \coloneqq zx \in \langle I \rangle$. The latter implies that $(c/1) y \neq 0/1$.
  By definition of $\A_I$ there are $b\in \A$ and $s\in R\setminus I$ such that $y = b/s$.
  We have shown that $(ab)/(qs) = (a/q)(b/s) = 0/1$ for all $a\in I$, $q\in R \setminus I$,
  and that $(cb)/s = (c/1)(b/s) \neq 0/1$. The former means that $ab \in [I : \A]$ for all $a\in I$, i.e.\ $b\in I^\perp$,
  and the latter that $cb \notin [I:\A]$.
\end{proof}

\begin{corollary} \label{corollary:Ibot}
  Let $\A$ be a ring that is finitely generated as a module over its center $R \coloneqq \zentrum(\A)$.
  Then $\A$ is weakly left Jacobson if and only if $\A$ is Jacobson.
\end{corollary}
\begin{proof}
  If $\A$ is weakly left Jacobson then $\A$ clearly is Jacobson. Conversely, assume that $\A$ is Jacobson and let
  $I$ be a prime left ideal of $\A$. We show that $I$ is the intersection of all maximal left ideals of $\A$ that contain $I$.
  In order to do so consider any $c \in \A \setminus I$. The proof is complete once we have found a maximal
  left ideal $\maxLId$ of $\A$ that fulfils $I \subseteq \maxLId$ and $c \notin \maxLId$.

  By the previous Theorem~\ref{theorem:Ibot} there exists $b\in I^\perp$ such that $cb \notin [I:\A]$.
  As $[I:\A]$ is a prime ideal of $\A$ by Proposition~\ref{proposition:sim}, and as $\A$ is Jacobson by assumption,
  there exists a maximal left ideal $\tilde\maxLId$ of $\A$ that fulfils $[I:\A] \subseteq \tilde\maxLId$ and $cb \notin \tilde\maxLId$.
  We construct the simple $\A$-left module $M \coloneqq \A / \tilde\maxLId$ and write $[a] \in M$ for the equivalence class of an element $a\in \A$.
  Then $a [b] = [ab] = [0]$ for all $a \in I$ because $b\in I^\bot$ so that $ab \in [I:\A] \subseteq \tilde\maxLId$, and $c [b] = [cb] \neq [0]$.
  Note that this also implies that $[b] \neq [0]$. So the annihilator $\maxLId \coloneqq \ann_M([b]) = \{ a \in \A \mid a[b] = [0] \}$
  is a maximal left ideal of $\A$ because $M$ is simple (see Lemma~\ref{interpretationmax}), and $I \subseteq \maxLId$
  and $c \notin \maxLId$ hold as required.
\end{proof}

\begin{corollary}
  Let $\A$ be a ring that is finitely generated as a module over its center $R \coloneqq \zentrum(\A)$.
  If $R$ is Noetherian, then $\A$ is weakly left Jacobson.
\end{corollary}
\begin{proof}
  If $R$ is Noetherian, then $\A$ is Jacobson by \cite[Sec.~9.1.3]{mcr}, so $\A$ is weakly left Jacobson by the previous
  Corollary~\ref{corollary:Ibot}.
\end{proof}

\section{Some strongly Jacobson rings}
\label{sec5}
In this section we show that an Azumaya $R$-algebra is strongly left Jacobson if and only if $R$ is Jacobson.
This follows from Corollary  \ref{corollary:Ibot} and Proposition  \ref{proposition:semiprime2prime};
see Theorem \ref{theorem:stronglyLeftJacobson}.

Let the ring $\A$ be  a separable algebra over its center $R \coloneqq \zentrum(\A)$;
i.e.\ $\A$ is an Azumaya $R$-algebra. By \cite[Thm.~3.4 of Chap.~II]{demeyer} or \cite[Thm.~7.1.4]{ford}, $\A$ is a finitely generated
projective $R$-module and the morphism of $R$-algebras
\begin{equation}
  \label{eq:Xidef}
  \Xi_\A \colon \A \otimes_R \A^\op \to \End_R(\A),\quad a\otimes b \mapsto \Xi_\A(a\otimes b)
\end{equation}
defined by
\begin{equation}
  \Xi_\A(a\otimes b)(c) \coloneqq a c b\quad\quad\text{for $a,b,c\in \A$}
\end{equation}
is an isomorphism.
Here $\A^\op$ denoted the opposite ring of $\A$
and $\End_R(\A)$ is the $R$-algebra of $R$-module endomorphisms on $\A$.

\begin{lemma} \label{lemma:Xi}
  Let $\A$ be an Azumaya algebra with center $R \coloneqq \zentrum(\A)$ and let $(v_i,\omega_i)_{i=1,\ldots,d}$
  be a dual basis of $M \coloneqq \A$ considered as a (finitely generated projective) $R$-module;
  i.e.\ $v_1, \dots,v_d \in M$ and $\omega_1, \dots, \omega_d \in \Hom_R(M,R)$ are such that $\sum_{i=1}^d \omega_i(x)v_i = x$
  for every $x \in M$.
  Then there exist $\rho \in \NN$ and $e_i^\alpha, f_i^\alpha \in \A$ for $i\in \{1,\dots,d\}$, $\alpha \in \{1,\dots,\rho\}$
  such that  for all $i\in \{1,\dots,d\}$ and  $x \in \A$ we have
  \begin{equation}
    \label{eq:Xi}
    \sum_{\alpha = 1}^\rho e_i^\alpha x f_i^\alpha = \omega_i(x)
  \end{equation}
\end{lemma}
\begin{proof}
  Every  functional $\omega_i \colon \A \to R \subseteq \A$ with $i\in \{1,\dots,d\}$ describes an element of $\End_R(\A)$.
  Since $\Xi_\A \colon \A \otimes_R \A^\op \to \End_R(\A)$ is surjective, the preimages of $\omega_i$
  under $\Xi_\A$ provide elements $e_i^\alpha, f_i^\alpha \in \A$ for $i\in \{1,\dots,d\}$ and $\alpha \in \{1,\dots,\rho\}$,
  with $\rho \in \NN$ sufficiently large, such that \eqref{eq:Xi} holds.
\end{proof}

A submodule $N$ of a module $M$ over a commutative ring $R$ is called \emph{prime} if the following holds:
Whenever elements $r\in R$ and $m\in M$ fulfil $rm \in N$, then $r M \subseteq N$ or $m \in N$.
In \cite[Def.~2.1]{aryapoor}, a submodule $N$ of a module $M$  is called \emph{semiprime} if the following holds:
Whenever an element $m \in M$ fulfils $m \in \sum(N:m)M$ we have $m \in N$;
here $(N:m)$ denotes the ideal $(N:m) \coloneqq \{r\in R\mid rm \in N\}$ of $R$.
By \cite[Thm.~2.6]{aryapoor}, every semiprime submodule of a finitely generated $R$-module is an intersection of prime submodules.

\begin{lemma} \label{lem:altsemiprime}
  Suppose $M$ is a finitely generated projective $R$-module. Pick a dual basis $(v_i,\omega_i)_{i=1,\ldots,d}$ of $M$.
  A submodule $N$ of $M$ is semiprime if and only if for every $x \in M$ that fulfils $\omega_i(x) \in (N:x)$ for all $i\in \{1,\dots,d\}$ we have $x \in N$.
\end{lemma}

\begin{proof} 
  Suppose that $N$ is a semiprime submodule of $M$. Pick $x \in M$ such that $\omega_i(x) \in (N:x)$ for all $i\in\{1,\dots,d\}$.
  Then it follows that $x=\sum_{i=1}^d \omega_i(x) v_i \in \sum(N:x)M$. Since $N$ is semiprime, this implies $x \in N$.

  Conversely, consider a submodule $N$ of $M$ such that for any $x \in M$ it follows from $\omega_i(x) \in (N:x)$ for all $i\in \{1,\dots,d\}$ that $x \in N$.
  To prove that $N$ is semiprime, take any $x = \sum_{j=1}^k r_j m_j \in M$ with $k\in \NN_0$ and $r_j \in (N:x)$, $m_j \in M$ for $j\in\{1,\ldots,k\}$.
  Since $(N:x)$ is an ideal in $R$ we have $\omega_i(x)=\sum_{j=1}^k r_j \omega_i(m_j) \in (N:x)$ for all $i\in \{1,\dots,d\}$,
  so $x \in N$ by assumption.
\end{proof}

\begin{proposition} \label{proposition:semiprime2prime}
  Let the ring $\A$ be an Azumaya algebra over its center $R \coloneqq \zentrum(\A)$.
  Then  every semiprime left ideal of $\A$ is an intersection of prime left ideals of $\A$.
\end{proposition}
\begin{proof}
  Let $I$ be a semiprime left ideal of $\A$ and consider an arbitrary element $c\in \A \setminus I$. In order to complete
  the proof we construct a prime left ideal $J$ of $\A$ that fulfils $I \subseteq J$ and $c\notin J$.

  Let  $(v_i,\omega_i)_{i=1,\ldots,d}$ be a dual basis of $\A$ considered as an $R$-module.
  The first step is to show that $I$ is also semiprime as a submodule of $\A$.
  So assume that $x \in \A$ satisfies $\omega_i(x) \in (I:x)$ for all $i\in\{1,\ldots,d\}$.
  Note that $\omega_i(x) \in (I:x)$ means that $\omega_i(x)x \in I$.
  Then $x a x = \sum_{i=1}^d \omega_i(x) v_i a x = \sum_{i=1}^d v_i a \, \omega_i(x)x \in I$ for all $a\in \A$, which implies $x \in I$
  because $I$ is a semiprime left ideal of $\A$. By Lemma \ref{lem:altsemiprime} this means that $I$ is indeed a semiprime submodule of $\A$.

  By \cite[Thm.~2.6]{aryapoor} there exists a prime $R$-submodule $N$ of $\A$
  such that $I \subseteq N$ and $c \notin N$. Set $J \coloneqq \{ x \in \A \mid \A x \subseteq N \}$.
  It is easy to check that $J$ is a left ideal of $N$ and $J\subseteq N$ because $\A$ has a unit. In particular $c\notin J$.
  Moreover, $I \subseteq J$ because $\A x \subseteq I \subseteq N$ for all $x\in I$. It only remains to show
  that $J$ is prime as a left ideal of $\A$:
  
  Let $\rho \in \NN$ and $e_i^\alpha, f_i^\alpha \in \A$ for $i\in \{1,\dots,d\}$, $\alpha \in \{1,\dots,\rho\}$ fulfil
  formula \eqref{eq:Xi} of Lemma \ref{lemma:Xi}.
  Pick any $x \in \A$ and $x' \in \A \setminus J$ such that $xax' \in J$ for all $a \in \A$.
  As $x' \notin J$ there is $a' \in \A$ such that $a' x' \notin N$.
  On the other hand, for each $i\in\{1,\ldots,d\}$ we have
  $\omega_i(x)a'x'=\sum_{\alpha = 1}^\rho e_i^\alpha x f_i^\alpha a' x'  \in J \subseteq N$.
  Since $N$ is a prime $R$-submodule of $\A$ and $a'x' \notin N$ we have $\omega_i(x)\A \subseteq N$ for all $i\in\{1,\ldots,d\}$.
  It follows that $ax=a\sum_{i=1}^d \omega_i(x) v_i = \sum_{i=1}^d \omega_i(x) a v_i\in N$ for all $a\in \A$, so $x\in J$.
\end{proof}

The following technical Lemma will be needed in the proofs of 
Theorem \ref{theorem:stronglyLeftJacobson} and Proposition \ref{proposition:findem}.

\begin{lemma}
\label{lemma:PI}
Let $\A$ be a ring which is finitely generated as a module over its center.
Then every left-primitive ideal of $\A$ is maximal. In particular,
if $\maxLId$ is a maximal left ideal of $\A$, then the (two-sided)
      ideal $[\maxLId : \A]$ of $\A$ from \eqref{eq:quotient} is maximal.
\end{lemma}

\begin{proof}
  Let $\maxLId$ be a maximal left ideal of $\A$.
  Note that $[\maxLId : \A]$ is the annihilator of the simple left $\A$-module $\A / \maxLId$, so
  $[\maxLId : \A]$ is left-primitive. Since $\A$ is a PI ring by \cite[Lemma~1.19]{formanek},
  every left-primitive ideal of $\A$ is maximal by Kaplansky's Theorem \cite[Thm.~4.13]{formanek};
  in particular $[\maxLId : \A]$ is a maximal ideal of $\A$.
\end{proof}

By \cite[Cor.~3.7 of Chap.~II]{demeyer} or \cite[Cor.~7.1.7]{ford} there is a bijective correspondence between the (two-sided) ideals of
an Azumaya algebra $\A$ and the ideals of its center $R \coloneqq \zentrum(\A)$. It is given by
\begin{equation}
  \label{eq:correspondenceOfIdeals}
  I \mapsto I \cap R \quad\quad\text{and}\quad\quad J \mapsto {\textstyle\sum} J \A
\end{equation}
for (two-sided) ideals $I$ of $\A$ and ideals $J$ of $R$, respectively.
It is clear that the relation $\subseteq$ is preserved in both directions, and consequently also maximal ideals are preserved.
As observed in \cite[Thm.~8.17(b)]{formanek} the correspondence also preserves prime ideals.

\begin{theorem}\label{theorem:stronglyLeftJacobson}
  Let $\A$ be an Azumaya algebra over its center $R \coloneqq \zentrum(\A)$.
  Then $\A$ is strongly left Jacobson if and only if $R$ is Jacobson.
\end{theorem}
\begin{proof}
  By the correspondence \eqref{eq:correspondenceOfIdeals}, $R$ is Jacobson (i.e.\ every prime ideal of $R$ is an intersection
  of maximal ideals of $R$) if and only if every prime ideal of $\A$ is an intersection of maximal ideals of $\A$.
  The latter in turn is true if and only if every prime ideal of $\A$ is an intersection of left-primitive ideals
  (because the left-primitive ideals in $\A$ coincide with the maximal ideals by Lemma \ref{lemma:PI}).
  By definition this is true if and only if $\A$ is Jacobson.
  By Corollary \ref{corollary:Ibot} and Proposition \ref{proposition:semiprime2prime},
  $\A$ is Jacobson if and only if $\A$ is strongly left Jacobson.
\end{proof}

We also make the following observation:

\begin{proposition}
\label{proposition:findem}
  Let $\A$ be an Azumaya algebra over its center $R \coloneqq \zentrum(\A)$
  and let $\maxLId$ be a maximal left ideal of $\A$. 
 Then $\maxLId \cap R$ is a maximal ideal of $R$ and 
$\A / \maxLId$ is a finite-dimensional vector space over $R/(\maxLId \cap R)$.
\end{proposition}

\begin{proof}
Since $\maxLId$ is a maximal left ideal of $\A$,
 $[\maxLId : \A]$ is a maximal  two-sided ideal of $\A$ by Lemma \ref{lemma:PI}.
   Since the correspondence \eqref{eq:correspondenceOfIdeals} preserves maximal ideals,
    $[\maxLId : \A] \cap R$ is a maximal ideal of $R$.
  By Proposition \ref{proposition:sim}, $[\maxLId : \A] \cap R = \maxLId \cap R$
  which implies that  $\maxLId \cap R$ is a maximal ideal of $R$.
    As $\A$ is a finitely generated module over $R$, the dimension of $\A / \maxLId$
    over the field $R/(\maxLId \cap R)$ is finite.
\end{proof}

\section{Polynomials with coefficients in algebras}
\label{sec6}

The aim of this section is to show that every $\FF$-algebra of the form 
$\A[x_1,\ldots,x_n]\coloneqq \A \otimes_\FF \FF[x_1,\dots,x_n]$, where $\A$ is a
finite-dimensional algebra over a field $\FF$ and $x_1,\ldots,x_n$ are central variables, 
satisfies the left Nullstellensatz; see Theorem  \ref{theorem:algebraValuedPolynomials}.
The proof consists of three steps.
For simple $\A$  we use Theorem~\ref{theorem:stronglyLeftJacobson}, then we prove 
 it for semisimple $\A$ and finally for  general $\A$. 
  In parallel we also prove that every maximal left ideal of $\A[x_1,\ldots,x_n]$
 has a special form, see Proposition \ref{proposition:Jxiv}.
 The  cases $\A=\HH$ and $\A=M_k(\CC)$ were already discussed in \cite[Theorems 1.2 and 1.4]{c3}.

%In order to simplify our notation we identify $\A$ with the subalgebra $\A \otimes 1$ of $\A[x_1,\dots,x_n]$.
If $\A$ is a finite-dimensional and simple $\FF$-algebra, then its center $\EE \coloneqq \zentrum(\A)$ is a finite (hence algebraic) field
extension of $\FF$. Write $\bar \FF$ for the algebraic closure of $\FF$, then $\bar \FF$ is also the algebraic closure
of $\EE$, so by \cite[Prop.~4.5.1.]{ford}, there exists $k\in \NN$ such that $\A \otimes_\EE \bar \FF \cong \Mat_k(\bar\FF)$
where $\Mat_k(\bar\FF)$ denotes the $\bar\FF$-algebra of square matrices of size $k$.

For $\xi \in \bar \FF^n$ an evaluation map $\A[x_1,\dots,x_n] \ni a \mapsto a(\xi) \in \Mat_k(\bar\FF)$
can be defined by
\begin{equation}
 a(\xi) \coloneqq \sum_\nu a_\nu \xi^\nu \in \A \otimes_\EE \bar \FF \cong \Mat_k(\bar\FF)
\end{equation}
for all $a = \sum_\nu a_\nu x^\nu \in \A[x_1, \dots, x_n]$ with coefficients $a_\nu\in \A$ (using standard
multiindex notation). This evaluation map is clearly is an $\EE$-algebra morphism, and therefore
\begin{equation}
  J_{\xi,v} \coloneqq \bigl\{ a\in \A[x_1,\dots,x_n] \mid a(\xi) v = 0 \bigr\}
\end{equation}
is a left ideal of $\A[x_1,\dots,x_n]$ for every $v\in \bar\FF^k$. Note that $1 \in J_{\xi,v}$ if and only if $v=0$.

\begin{proposition} \label{proposition:simple}
  Let $\FF$ be a field, $\A$ a finite-dimensional simple $\FF$-algebra, and $n\in \NN$.
  Then $\A[x_1,\dots,x_n]$ is strongly left Jacobson, every maximal left ideal $\maxLId$ of $\A$
  has finite codimension over $\FF$, and there exist a point $\xi \in \bar \FF^n$ and 
  a vector $v\in \bar \FF^k \setminus \{0\}$   such that $\maxLId = J_{\xi,v}$.
\end{proposition}
\begin{proof}
  Let $\EE$ be the center of $\A$ so that the center of $\A[x_1,\dots,x_n]$ is $R \coloneqq \zentrum\bigl( \A[x_1,\dots,x_n] \bigr) = \EE[x_1,\ldots,x_n]$.
  Since $\A$ is central simple over $\EE$ it is also Azumaya by \cite[Cor.~4.5.4.]{ford}. By \cite[Cor.~4.3.2]{ford}
  it follows that $\A[x_1,\dots,x_n] \cong \A \otimes_\EE R$ is Azumaya over $R$.
  By Hilbert's Nullstellensatz, the polynomial ring $R$
  is Jacobson (see e.g. \cite[Thm.~4.9]{eisenbud} or \cite[Cor.~on p.~138]{og})
  and every maximal ideal of $R$ has finite codimension over $\EE$ 
  (see e.g \cite[Cor.~7.10]{atiyah} or \cite[Cor.~on p.~139]{og}).
Theorem~\ref{theorem:stronglyLeftJacobson} now shows that $\A[x_1,\dots,x_n]$ is strongly left Jacobson.

  Moreover, if $\maxLId$ is a maximal left ideal of $\A[x_1,\dots,x_n]$, then by Proposition \ref{proposition:findem},
  $\maxLId \cap R$ is a maximal ideal of $R$ and by Hilbert's Nullstellensatz it
   has finite codimension over $\EE$ in $R$.
  Write $\KK \coloneqq R/(\maxLId \cap R)$, then $\KK$ is a finite field extension of $\EE$.
  By Proposition \ref{proposition:findem}, $\A[x_1,\dots,x_n]/\maxLId$ is finite dimensional over $\KK$ and therefore
  \begin{align*}
    \dim_\FF \A[x_1,\dots,x_n]/\maxLId = \dim_\FF  \EE \, \dim_\EE \KK \, \dim_\KK \A[x_1,\dots,x_n]/\maxLId  < \infty.
   \end{align*}

    As $\KK$ is a finite (hence algebraic) extension of $\FF$, it embedds into the algebraic closure $\bar\FF$.
  Set $\xi \coloneqq (\xi_1, \dots, \xi_n) \coloneqq \bigl( [x_1], \dots, [x_n] \bigr) \in \KK^n \subseteq \bar\FF^n$
  where $[\argument] \colon R \to \KK$ is the canonical projection onto the quotient, which is an $\EE$-algebra morphism.
  Then $\maxLId \cap R = \ker [\argument] = \{ p\in R \mid p(\xi)=0 \}$ because $p(\xi) = [p(x_1,\dots,x_n)] = [p]$ for all $p \in R$.

  We claim that the two-sided ideal $\maxId  := [\maxLId : A[x_1,\dots,x_n]]$
  is equal to $ \{ a \in \A[x_1,\dots,x_n] \mid a(\xi) = 0 \}$.
  As $\A[x_1,\dots,x_n]$ is Azumaya over $R$  and  $\maxId \cap R = \maxLId \cap R$
  we have   $\maxId = \sum(\maxLId \cap R)\A[x_1,\dots,x_n]$ by the correspondence 
   \eqref{eq:correspondenceOfIdeals}. Therefore
  $\maxId \subseteq \{ a \in \A[x_1,\dots,x_n] \mid  a(\xi) = 0 \}$.
  As $\maxId$ is a maximal two-sided ideal by Lemma \ref{lemma:PI}, we have equality.
 % (Note that the right-hand side is a two-sided ideal.)
  
  Write $\maxLId^\bot \coloneqq \{ b\in \A[x_1,\dots,x_n] \mid \maxLId\, b \subseteq \maxId \}$ like in Theorem~\ref{theorem:Ibot},
  then $\maxLId = \{ a\in \A[x_1,\dots,x_n] \mid a \,\maxLId^\bot \subseteq \maxId \}$, or equivalently,
  \begin{equation*}
    \maxLId = \{ a\in \A[x_1,\dots,x_n] \mid a(\xi)\, b(\xi)= 0~\textup{for all $b\in \maxLId^\bot$} \}
    .
  \end{equation*}
  As $1 \notin \maxLId$ there is $b\in \maxLId^\bot$ such that $b(\xi) \neq 0$.
  Let $v\in \bar\FF^k \setminus \{0\}$ be any non-zero column of $b(\xi) \in \Mat_k(\bar\FF)$,
  then $a(\xi) v = 0$ for all $a\in \A[x_1,\dots,x_n]$, so $\maxLId \subseteq J_{\xi,v}$.
  By maximality of $\maxLId$ it follows that $\maxLId = J_{\xi,v}$.
\end{proof}

\begin{example}
Note, however, that not every left ideal of $\A[x_1,\dots,x_n]$ of the form $J_{\xi,v}$ is maximal.
Namely, pick $\A =M_2(\QQ)$, $\xi=(0,\ldots,0) \in \QQ^n$ and $v=[ 1 \ \sqrt{2}]^T \in \bar{\QQ}^2$.
Since $\sqrt{2}$ is irrational, $a(\xi)v=0$ is equivalent to $a(\xi)=0$ for every $a \in \A[x_1,\ldots,x_n]$.
To show that $J_{\xi,v}$ is not a maximal left ideal pick any $0 \ne u \in \QQ^2$ 
and note that $a(\xi)u=0$ need not imply $a(\xi)=0$.
Therefore $J_{\xi,v} \subsetneq J_{\xi,u} \subsetneq \A[x_1,\dots,x_n]$. 
\end{example}

\begin{proposition} \label{proposition:directsum}
  Consider rings $\A_1,\ldots,\A_k$ with $k\in \NN$ and let $\A \coloneqq \A_1 \oplus \dots \oplus \A_k$.
  For any $j \in \{1,\dots,k\}$ and any left ideal $I_j$ of $\A_j$ we write
  \begin{equation}
    \label{eq:directsum:hat}
    \hat I_j \coloneqq \A_1 \oplus \dots \oplus \A_{j-1} \oplus I_j \oplus \A_{j+1} \oplus \dots \oplus \A_k \subseteq \A
    .
  \end{equation}
  Then a left ideal $\maxLId$ of $\A$ is maximal if and only if there exist $j\in \{1,\dots,k\}$ and a maximal ideal $\maxLId_j$ of $\A_j$
  such that $\maxLId = \hat \maxLId_j$.
  Moreover, if all $\A_1, \dots, \A_k$ are strongly left Jacobson, then $\A$ is strongly left Jacobson.
\end{proposition}
\begin{proof}
  Write $p_j \coloneqq 0 \oplus \dots \oplus 0 \oplus 1 \oplus 0 \oplus \dots \oplus 0 \in \A_j \subseteq \A$
  for $j\in \{1,\dots,k\}$,
  where $1$ denotes the unit of $\A_j$ and is situated at position $j$. Clearly $\sum_{j=1}^k p_j$ is the unit of $\A$.
  Every left ideal $I$ of $\A$ decomposes as
  \begin{equation*}
    I= I_1 \oplus \dots \oplus I_k = \hat{I}_1 \cap \dots \cap \hat{I}_k
    ,
  \end{equation*}
  where, for each $j \in \{ 1,\ldots,k \}$, $I_j \coloneqq p_j I = I \cap \A_j$ is a left ideal of $A_j$ and $\hat I_j$ is a left ideal of $\A$.
  In particular if $I = \maxLId$ is a maximal left ideal of $\A$, then by maximality of $\maxLId$, $\hat\maxLId_j = \A$
  for all but one $j\in \{1,\dots,k\}$, so $\maxLId = \hat\maxLId_j$ for this remaining $j\in \{1,\dots,k\}$.
  Conversely, if $\naxLId_j$ is any maximal left ideal of $\A_j$
  for some $j\in \{1,\dots,k\}$, then it is clear that $\hat\naxLId_j$ is a maximal left ideal of $\A$.

  Let $I$ be a semiprime left ideal of $\A$. For any $j\in \{1,\dots,k\}$, the left ideal $I_j$ of $\A_j$ is also semiprime:
  indeed, if $a\in \A_j$ fulfils $a \A_j a \subseteq I_j$, then $a \A a = a p_j \A p_j a \subseteq I_j \subseteq I$
  because $p_j \A p_j \subseteq \A_j$, so $a \in I$ because $I$ is semiprime, and therefore $a \in I_j$.
  Now consider any element $\hat c = c_1 \oplus \dots \oplus c_k \in \A \setminus I$ with components $c_j \in \A_j$ for $j\in \{1,\dots,k\}$.
  Then there exists $j\in \{1,\dots,k\}$ such that $c_j \in \A_j \setminus I_j$, and if $\A_j$ is strongly left Jacobson,
  then there is a maximal left ideal $\maxLId_j$ of $\A_j$ such that $I_j \subseteq \maxLId_j$ and $c_j \notin \maxLId_j$.
  In this case $\hat\maxLId_j$ is a maximal left ideal of $\A$ that fulfils $I \subseteq \hat\maxLId_j$ and $\hat c \notin \hat\maxLId_j$.
\end{proof}

\begin{corollary} \label{corollary:directsum}
  Let $\FF$ be a field, $\A$ a finite-dimensional semisimple $\FF$-algebra, and $n\in \NN$.
  Then $\A[x_1,\dots,x_n]$ is strongly left Jacobson.
  Moreover, every maximal left ideal of $\A[x_1,\dots,x_n]$ has finite codimension over $\FF$.  
\end{corollary}
\begin{proof}
  As $\A$ is semisimple there are $k\in \NN$ and simple $\FF$-algebras $\A_1, \dots, \A_k$ such that
  $\A \cong \A_1 \oplus \dots \oplus \A_k$. As $\otimes$ is distributive over $\oplus$,
  $\A[x_1,\dots,x_n] \cong \A_1[x_1,\dots,x_n] \oplus \dots \oplus \A_k[x_1,\dots,x_m]$.
  By Proposition~\ref{proposition:simple} each $\FF$-algebra $\A_j[x_1,\dots,x_n]$ with $j\in \{1,\dots,k\}$
  is strongly left Jacobson and every maximal left ideal $\maxLId_j$ of $\A_j[x_1,\dots,x_n]$ has finite codimension.
Proposition~\ref{proposition:directsum} now shows that $\A[x_1,\dots,x_n]$ itself is strongly left Jacobson
  and that every maximal left ideal of $\A[x_1,\dots,x_n]$ is of the form $\hat\maxLId_j$ as in \eqref{eq:directsum:hat}
  for some $j\in \{1,\dots,m\}$ and some maximal left ideal $\maxLId_j$ of $\A_j[x_1,\dots,x_n]$.
  Clearly the codimension   of $\hat\maxLId_j$ in $\A[x_1,\dots,x_n]$ equals the codimension 
  of $\maxLId_j$ in $\A_j[x_1,\dots,x_n]$, which is finite.
\end{proof}

We need only one more technical Lemma before we can proof our main theorem:

\begin{lemma} \label{lemma:radcorr}
  Let $\A$ be a finite-dimensional algebra over some field $\FF$ and $\B \coloneqq \A/\rad \A$
  with $\rad \A:=\rad(\{0\})$ the Jacobson radical of  $\A$.
% as in Section~\ref{sec2}.
  Write $\tilde{\pi} \colon \A \to \B$ for the canonical projection and define
  \begin{equation}
    \pi \coloneqq \tilde\pi \otimes \id_{\FF[x_1,\dots,x_n]} \colon \A[x_1,\dots,x_n] \to \B[x_1,\dots,x_n]
    .
  \end{equation}
  Then every semiprime left ideal of $\A[x_1,\ldots,x_n]$ contains $\ker \pi$.
\end{lemma}

\begin{proof}
  Consider a semiprime left ideal $I$ of $\A[x_1,\dots,x_n]$ and let $J \coloneqq \bigl[I \colon \A[x_1,\dots,x_n]\bigr] =
  \{ a \in \A[x_1,\dots,x_n] \mid a\A[x_1,\dots,x_n] \subseteq I\}$
  be the (two-sided) ideal of $\A[x_1,\dots,x_n]$ like in \eqref{eq:quotient}.
  Then $J$ is again semiprime: Indeed, if $a\in \A[x_1,\dots,x_n]$ fulfils $a c a \in J$ for all $c\in\A[x_1,\dots,x_n]$,
  then this means that $acad \in I$ for all $c,d\in\A[x_1,\dots,x_n]$, and therefore also $adcad \in I$ for all $c,d\in\A[x_1,\dots,x_n]$.
  But this implies $ad \in I$ for all $d\in\A[x_1,\dots,x_n]$ because $I$ is semiprime, so $a \in J$.

  Next we show that $J \cap \A$ is a semiprime (two-sided) ideal of $\A$:
  It is clear that $J \cap \A$ is a two-sided ideal of $\A$.
  Let $d\in \NN$ be the dimension of $\A$ over $\FF$ and let $b_1,\dots,b_d$ be an $\FF$-basis of $\A$.
  If $a\in \A$ fulfils $a c a \in J \cap \A$ for all $c\in \A$, then also $a (\sum_{i=1}^d p_i b_i) a = \sum_{i=1}^d p_i ab_ia \in J$
  for all elements $\sum_{i=1}^d p_i b_i \in \A[x_1,\dots,x_n]$ with coefficients $p_1,\dots,p_d \in \FF[x_1,\dots,x_n]$,
  so $a\in J$ because $J$ is a semiprime ideal of $\A[x_1,\dots,x_n]$.

  By \cite[Thm.~4.12]{lam}, $\rad \A$ is a nilpotent ideal; i.e.\ $(\rad\A)^n=\{0\} \subseteq  J \cap \A$ for some $n$.
  Since  $J \cap \A$ is semiprime, it follows that $\rad\A  \subseteq J \cap \A$.
  Using standard multiindex notation, any element $a \in \A[x_1,\ldots,x_n]$ can be expanded as a finite sum $a=\sum_\nu a_\nu x^\nu$
  with suitable coefficients $a_\nu \in \A$.
  If $a \in \ker \pi$, then $0=\pi(a)=\sum_\nu \tilde\pi(a_\nu) x^\nu$, so $\tilde\pi(a_\nu)=0$ for all multiindices $\nu$.
  Therefore $a_\nu \in \rad\A \subseteq  J\cap \A$ for all multiindices $\nu$, hence $a \in J$.
  This shows that $\ker \pi \subseteq J \subseteq I$.
\end{proof}

\begin{theorem} \label{theorem:algebraValuedPolynomials}
  Let $\A$ be a finite-dimensional algebra over some field $\FF$ and $n\in \NN$.
  Then $\A[x_1,\dots,x_n]$ is strongly left Jacobson.
  Moreover, every maximal left ideal of $\A[x_1,\dots,x_n]$ has finite codimension over $\FF$.  
\end{theorem}
\begin{proof}
  Recall that if we factor a finite-dimensional algebra with its Jacobson radical
  then we get a semisimple algebra; see \cite[Thm.~4.14]{lam}.
  Corollary~\ref{corollary:directsum} therefore applies to the algebra $\B[x_1,\dots,x_n]$
  where $\B=\A/\rad \A$.
  So $\B[x_1,\dots,x_n]$ is strongly left Jacobson and that every maximal left ideal  of $\B[x_1,\dots,x_n]$   has finite codimension over $\FF$.

  As the kernel of the surjective algebra morphism $\pi \colon \A[x_1,\dots,x_n] \to \B[x_1,\dots,x_n]$
  from Lemma~\ref{lemma:radcorr} is contained in every semiprime left ideal of $\A[x_1,\ldots,x_n]$,
  Lemma \ref{lemma:homomorphicimage} shows that the image and preimage under $\pi$ give a bijective correspondence between
  the semiprime (resp.\ prime, maximal) left ideals of $\A[x_1,\dots,x_n]$ and semiprime (resp.\ prime, maximal) left ideals of $\B[x_1,\dots,x_n]$.
  It follows that also $\A[x_1,\dots,x_n]$ is strongly left Jacobson and that every maximal left ideal of
  $\A[x_1,\dots,x_n]$ has finite codimension over $\FF$.
\end{proof}

Finally we want to reformulate the above noncommutative Nullstellensatz in a more explicit way. 
So let $\A$ be a finite-dimensional algebra over a field $\FF$ and $n\in \NN$. 
Let $\theta \colon \A \to \Si$ be a surjective $\FF$-algebra morphism onto
a simple $\FF$-algebra $\Si$, then define
\begin{equation}
  \label{eq:Theta}
  \Theta \coloneqq \theta \otimes_\FF \id_{\FF[x_1,\dots,x_n]} \colon \A[x_1,\dots,x_n] \to \Si[x_1,\dots,x_n]
  .
\end{equation}
Write $\EE \coloneqq \zentrum(\Si)$ for the center of $\Si$, then
like in the discussion at the beginning of this section,
$\Si \otimes_\EE \bar \FF \cong \Mat_k(\bar\FF)$ for some $k\in \NN$.
Let $\xi \in \bar \FF^n$ and $v\in \bar \FF^k$, then
\begin{equation}
  \label{eq:Jxiv}
  J_{\theta,\xi,v} \coloneqq \bigl\{ a\in\A[x_1,\dots,x_n] \mid \Theta(a)(\xi) v = 0 \bigr\}
  %=\Theta^{-1}(J_{\xi,v}) 
\end{equation}
is a left ideal of $\A[x_1,\dots,x_n]$ and $1 \in J_{\theta,\xi,v}$ if and only if $v=0$.

\begin{proposition} \label{proposition:Jxiv}
  Let $\A$ be a finite-dimensional algebra over some field $\FF$ and $n\in \NN$.
  Then for every maximal left ideal $\naxLId$ of $\A[x_1,\dots,x_n]$  there exist 
  a surjective $\FF$-algebra morphism   $\theta \colon \A \to \Si$ onto a simple $\FF$-algebra $\Si$, 
  a point $\xi \in \bar \FF^n$,    and a vector $v\in \bar \FF^k$   such that $\naxLId = J_{\theta,\xi,v}$.
\end{proposition}
\begin{proof}
The $\FF$-algebra $\B=\A/\rad \A$ is semisimple, so it decomposes into a direct sum
of simple $\FF$-algebra $\B_1 \oplus \ldots \oplus \B_k$. Let $\tilde{\pi} \colon \A \to \B$
and $\tilde{\tau}_i \colon  \B \to \B_i$ be the canonical maps. The corresponding maps
on polynomial rings will be denoted by $\pi$ and $\tau_i$ respectively.

By the proof of Theorem  \ref{theorem:algebraValuedPolynomials}, 
every maximal left ideal of $\A[x_1,\dots,x_n]$ is of the form
$\pi^{-1}(\maxLId)$ where $\maxLId$ is a maximal left ideal of $\B[x_1,\dots,x_n]$.
By the proof of Corollary  \ref{corollary:directsum}, there exist $j \in \{1,\ldots,k\}$
and a maximal left ideal  $\maxLId_j$ of $\B_j[x_1,\dots,x_n]$ such that
$\maxLId=\tau_j^{-1}(\maxLId_j)$. 
Write $\Si=\B_j$,  $\Theta = \tau_j \circ \pi$ and $\theta=\tilde{\tau}_j \circ \tilde{\pi}$.
By Proposition \ref{proposition:simple}, $\maxLId_ j= J_{\xi,v}$ for some $\xi$ and $v$.
It follows that $\pi^{-1}(\maxLId)=\Theta^{-1}(\maxLId_j)=\Theta^{-1}(J_{\xi,v})= J_{\theta,\xi,v}$.
\end{proof}

\begin{corollary}
\label{cor:final}
  Let $\A$ be a finite-dimensional algebra over some field $\FF$ and $n\in \NN$.
  Then every semiprime left ideal $I$ of $\A[x_1,\dots,x_n]$ is the intersection of all left ideals of $\A[x_1,\dots,x_n]$
  of the form $J_{\theta,\xi,v}$ as in \eqref{eq:Jxiv} that contain $I$.
\end{corollary}
\begin{proof}
  Let $I$ be a semiprime left ideal of $\A[x_1,\dots,x_n]$.
  By Theorem~\ref{theorem:algebraValuedPolynomials} and Proposition~\ref{proposition:Jxiv}, $I$ is the intersection of
  \emph{some} left ideals of the form $J_{\theta,\xi,v}$ as in \eqref{eq:Jxiv},
  and therefore $I$ is the intersection of
  \emph{all} left ideals of the form $J_{\theta,\xi,v}$ as in \eqref{eq:Jxiv} that contain $I$.
\end{proof}

\bibliographystyle{plainurl}

\normalfont

\end{document}